\documentclass[11pt]{amsart}
\usepackage{amssymb,amsmath,hyperref,multirow,epsfig,float,bbm}
\usepackage[headings]{fullpage}
\usepackage[usenames,dvipsnames]{color}

\usepackage[symbol]{footmisc}

\newtheorem{theorem}{Theorem}[section]
\newtheorem{example}[theorem]{Example}
\newtheorem{corollary}[theorem]{Corollary}

\newtheorem{lemma}[theorem]{Lemma}
\newtheorem{proposition}[theorem]{Proposition}
\theoremstyle{remark}
\newtheorem{remark}[theorem]{Remark}

\numberwithin{equation}{section}

\definecolor{blue}{rgb}{0,0,1}
\definecolor{red}{rgb}{1,0,0}
\definecolor{green}{rgb}{0,.6,.2}
\definecolor{purple}{rgb}{1,0,1}

\long\def\red#1\endred{\textcolor{red}{#1}}
\long\def\blue#1\endblue{\textcolor{blue}{#1}}
\long\def\purple#1\endpurple{\textcolor{purple}{ #1}}
\long\def\green#1\endgreen{\textcolor{green}{#1}}

\renewcommand{\subset}{\subseteq}

\renewcommand{\mid}{\, | \,}
\newcommand{\A}{\mathbb{A}}

\newcommand{\Aut}{\operatorname{Aut}}

\newcommand{\bs}{\backslash}

\newcommand{\comment}[1]{}
\newcommand{\cInd}{\operatorname{c-Ind}}

\newcommand{\C}{\mathbb{C}}

\newcommand{\e}{\varepsilon}
\newcommand{\E}{\mathcal{E}}

\newcommand{\F}{\mathbb{F}}

\newcommand{\g}{\gamma}

\newcommand{\Gal}{\operatorname{Gal}}
\newcommand{\GL}{\operatorname{GL}}

\newcommand{\Ind}{\operatorname{Ind}}
\newcommand{\I}{\mathcal{I}}

\newcommand{\Irr}{\operatorname{Irr}}

\newcommand{\leg}{\overwithdelims ()}

\newcommand{\mat}[4]{\begin{pmatrix} {#1} & {#2} \\ {#3} & {#4}
  \end{pmatrix}}

\newcommand{\meas}{\operatorname{meas}}
\renewcommand{\mod}{\text{ mod }}

\newcommand{\new}{\mathrm{new}}

\renewcommand{\O}{\mathfrak{o}}

\newcommand{\ol}{\overline}
\newcommand{\olG}{\overline{G}}

\newcommand{\p}{\mathfrak{p}}
\renewcommand{\P}{\mathfrak{p}_E}

\newcommand{\PGL}{\operatorname{PGL}}
\newcommand{\Q}{\mathbb{Q}}
\newcommand{\Qb}{\overline{\Q}}

\newcommand{\R}{\mathbb{R}}
\newcommand{\Res}{\operatorname{Res}}

\newcommand{\s}{\mathtt{s}}
\newcommand{\sg}[1]{\left<{#1}\right>}

\newcommand{\smat}[4]{\bigl(\begin{smallmatrix}{#1}&{#2}\\{#3}&{#4}\end{smallmatrix}\bigr )}

\newcommand{\tr}{\operatorname{tr}}

\newcommand{\w}{\omega}

\newcommand{\Z}{\mathbb{Z}}

\makeatletter
\def\@tocline#1#2#3#4#5#6#7{\relax
  \ifnum #1>\c@tocdepth 
  \else
    \par \addpenalty\@secpenalty\addvspace{#2}%
    \begingroup \hyphenpenalty\@M
    \@ifempty{#4}{%
      \@tempdima\csname r@tocindent\number#1\endcsname\relax
    }{%
      \@tempdima#4\relax
    }%
    \parindent\z@ \leftskip#3\relax \advance\leftskip\@tempdima\relax
    \rightskip\@pnumwidth plus4em \parfillskip-\@pnumwidth
    #5\leavevmode\hskip-\@tempdima
      \ifcase #1
       \or\or \hskip 1em \or \hskip 2em \else \hskip 3em \fi%
      #6\nobreak\relax
    \hfill\hbox to\@pnumwidth{\@tocpagenum{#7}}\par
    \nobreak
    \endgroup
  \fi}
\makeatother

\begin{document}
\title{Counting newforms with prescribed ramified supercuspidal components}
\author{Andrew Knightly}
\address{Department of Mathematics \& Statistics\\University of Maine
\\Neville Hall\\ Orono, ME  04469-5752, USA }

\author{Kimball Martin}
\address{Department of Mathematics $\cdot$ International Research and Education Center, Graduate School of Science, Osaka Metropolitan University, Osaka 558-8585, Japan}
\email{kimball@omu.ac.jp}
\address{Department of Mathematics, University of Oklahoma, Norman, OK 73019 USA}
\email{kimball.martin@ou.edu}

\begin{abstract} 
We give a formula for the number of newforms in $S_k^{\new}(N)$ that have prescribed ramified
supercuspidal components $\pi_p$ at a set $T$ of primes dividing $N$.
This dimension is given in terms of the trace of the Atkin--Lehner operator at 
$T$ on $S_k^{\new}(N)$. It depends only upon the weight, the level, the ramified 
quadratic extensions $E_p/{\mathbb Q}_p$ attached to the $\pi_p$, and the root number 
of each $\pi_p$.
The formula is completely explicit when $T$ consists of either a single prime or all prime
factors of $N$.
\end{abstract}

\maketitle
\date{\today}

\today
\thispagestyle{empty}

\section{Introduction}

A basic application of the trace formula is computing the dimensions of the spaces $S_k(N)$
of holomorphic cusp forms of weight $k$ and level $N$.  
There are various decompositions of $S_k(N)$ into smaller spaces, and it 
is natural to ask for dimensions of these spaces.  First, Atkin and Lehner decomposed 
$S_k(N)$ into a new space $S_k^\new(N)$ and an old space $S_k^{\mathrm{old}}(N)$ of forms coming from
lower levels.  Dimensions of new (and old) spaces can be computed recursively, 
and more explicit formulas were derived in \cite{GMar}.

Since each Hecke eigenform $f\in S_k^\new(N)$ determines an irreducible representation $\pi_p$ of 
$\PGL_2(\Q_p)$ for each prime $p$,
one can further decompose $S_k^\new(N)$ according to the possible local components 
$\pi_p$ at the primes $p \mid N$.  When $N$ is squarefree, $\pi_p$ is 
determined by the local root number $\varepsilon_p = \pm 1$, i.e., 
the local Atkin--Lehner sign.  In this setting, the second author \cite{KMar} gave dimension formulas for 
the subspaces of $S_k^\new(N)$ with any fixed collection of signs 
$\{ \varepsilon_p \}_{p \mid N}$.  While asymptotically all collections of 
signs are equally likely, there is actually a bias towards/against certain collections of 
local signs, as well as a bias towards the global root number 
$\varepsilon = (-1)^{k/2} \prod_{p \mid N} \varepsilon_p$ being $+1$.  The dimension 
formulas for specifying a single local or global root number were extended to general levels in 
\cite{M23,M25}.  Again, there is typically a bias towards one local or global root number 
over the other.

If $p^2 \mid N$, the possible local representations at $p$ are no longer determined 
by just the local root number.  The first author \cite{K} recently gave dimension 
formulas for spaces of forms whose local component $\pi_p$ at each $p \mid N$ is a fixed
supercuspidal of conductor $p^2$ or $p^3$, assuming for technical reasons that
$k > 2$.  This is essentially the most refined dimension formula one might hope for 
for levels such that $v_p(N) \in \{ 2, 3 \}$ for each $p \mid N$. 
(When $p^2 \parallel N$, there are other possible local representations at $p$, but they are 
not minimal.)   Here as well there is a bias towards/against certain collections of local representations.
The key ingredient in \cite{K} is the explicit computation of local elliptic orbital
integrals attached to matrix coefficients of the fixed supercuspidals.  The reason for the
restriction to conductors with small exponents is that these integrals become quite complicated when 
the exponent is large.

Our main result is Theorem \ref{dim} below, which is a dimension formula that allows 
for prescribed supercuspidals of any odd-power conductor (the ``ramified" supercuspidals).
 This formula is obtained without computing local orbital integrals explicitly, 
 by blending the approaches of \cite{K} and \cite{M23,M25}.
First, the simple trace formula in \cite{K} expresses the dimension as the sum of a main term and
a certain global elliptic orbital integral (see \eqref{dim1} below).  For a ramified supercuspidal $\pi_p$,
we show that the value of the local orbital integral is the product of the local root number $\e_p$
with a constant depending only on the ramified quadratic extension $E_p/\Q_p$ determined by $\pi_p$.
This, together with an analysis of the Galois orbit of $\pi_p$ given in Proposition \ref{prop:locgal},
 allows us to express the desired dimension in terms of the trace of an Atkin-Lehner operator on
the full space of newforms of the given level.
Traces of such operators were obtained in \cite{M23,M25}.
The method can be extended to other cases where the local root number appears in the orbital integral
(see Remark \ref{dimrem}(b) below).

\subsection{Main result}

To state the result precisely, fix a squarefree odd integer $T>1$, an integer 
$M\ge 1$ relatively prime to $T$, and an integer $N$ of the form
\begin{equation}\label{N}
N=M\prod_{p|T}p^{2r_p+1}
\end{equation}
with each $r_p\ge 1$ and $r_3=1$  if $3\mid T$.
For each $p \mid T$, fix a supercuspidal representation $\pi_p$ of $\PGL_2(\Q_p)$ 
of conductor $p^{2r_p+1}$.  It has an associated ramified quadratic extension $E_p/\Q_p$
that appears in the inducing data on both sides of the local Langlands correspondence.
We let $\pi_T =(\pi_p)_{p|T}$ denote this tuple of representations.

Let
\[S_k^\new(N;\pi_T)\subset S_k^\new(N)\]
denote the subspace spanned by the newforms that have local component $\pi_p$ at each $p\mid T$.
Because every irreducible admissible representation of $\PGL_2(\Q_p)$ with
conductor $p^{2r_p+1}$ must be supercuspidal (see the table at the end of \cite[\S1]{Sch}),
 we have
\begin{equation}\label{orth}
S_k^{\new}(N)=\bigoplus_{\pi_T}S_k^{\new}(N;\pi_T),
\end{equation}
an orthogonal direct sum.  The purpose of this paper is to compute the dimension of each 
subspace on the right-hand side when $k\ge 4$.  There are two striking qualitative
features of our result, namely for $N,k$ and $T\ge 5$ fixed as above:
\begin{itemize}
\item As $\pi_T$ varies over all tuples, there are only three possibilities for 
  $\dim S_k^{\new}(N;\pi_T)$, of the form $\I-\E, \I, \I+\E$ where $\I,\E>0$. 
The middle case occurs for all $\pi_T$ except those for which $E_p=\Q_p(\sqrt{-T})$ for all $p\mid T$.
\item In all cases where the bias $\E$ has been computed explicitly (i.e., when
$M=1$ or $T$ is prime), it depends only on $T$ and $M$ in \eqref{N}, and not the conductor 
exponents $2r_p+1$ or $k$. 
\end{itemize}

Our main result is the following.

\begin{theorem}\label{dim} 
Suppose $T\ge 5$ is odd, and fix a tuple $\pi_T=(\pi_p)_{p|T}$ as above.
Let $\e_{\pi_T}=\prod_{p|T}\e_p$ be the product of the root numbers of the 
$\pi_p$.   
 Define 
\[\Delta(\pi_T)=\begin{cases}1&\text{if }E_p=\Q_p(\sqrt{-T})\text{ for all $p \mid T$}\\
0&\text{otherwise.}\end{cases}\]
Let $k\ge 4$ be even and let $N$ be as in \eqref{N} (with $v_3(N)=3$ if $3\mid T$).
Then 
\begin{equation}\label{main}
\dim S_k^\new(N;\pi_T) = \frac{k-1}{12}\psi^\new(M)\prod_{p|T}\frac{p^2-1}2p^{r_p-1}
+\Delta(\pi_T)\e_{\pi_T}\frac{\tr(W_T|S_k^\new(N))}{\prod_{p|T}(p-1)p^{r_p-1}},
\end{equation}
 where $W_T=\prod_{p|T}W_p$ is the Atkin-Lehner operator at $T$ of level $N$, and
 $\psi^{\new}$ is the multiplicative function defined on prime powers by
\[\psi^{\new}(p^a) = \begin{cases} 
p(1-\frac1p)&\text{if }a=1\\
p^2(1-\frac1p-\frac1{p^2})&\text{if }a=2\\
p^a(1-\frac1p)(1-\frac1{p^2})&\text{if }a\ge 3.\end{cases}
\]

The above dimension formula is given explicitly in the two special cases $M=1$ 
and $T$ prime as follows.
Define constants $b_{T,e}$ according to the values in following table:

\begin{center}
\begin{tabular}{r|rrr}
$e$  & $T \equiv 1 \bmod 4$ & $T \equiv 3 \bmod 8$ & $T \equiv 7 \bmod 8$ \\
\hline 
$0$ & $1/2$ & $2$ & $1$ \\
$1,2$ & $-1/2$ & $-1$ & $0$ \\
$3$ & $1/2$ & $-3$ & $0$ \\
$4$ & $0$ & $3/2$ & $-1/2$ \\
$\ge 5$ & $0$ & $0$ & $0$ 
\end{tabular}
\end{center}

If $M=1$, then
\begin{equation}\label{M1}
\dim S_k^\new(N;\pi_T) = \frac{k-1}{12}\prod_{p|T}\frac{p^2-1}2p^{r_p-1}
+\Delta(\pi_T)\e(k,\pi_T) b_{T,0}\, h(-T)
\end{equation}
where $h(-T)$ is the class number of $\Q(\sqrt{-T})$ and $\e(k,\pi_T)=(-1)^{k/2}\e_{\pi_T}$ is
the common global root number of the newforms spanning $S_k^{\new}(N;\pi_T)$.

If $T=p\ge 5$ is prime, then
\begin{equation}\label{Tp}
\dim S_k^\new(N;\pi_T) = \frac{k-1}{12}\psi^\new(M)\frac{p^2-1}2p^{r_p-1}
+\Delta(\pi_T)(-1)^{k/2}\e_{\pi_T} b_{p,v_2(M)} \, \kappa_{-p}(M') \, h(-p), 
\end{equation}
where $M'$ is the odd part of $M$ and $\kappa_{-p}$ is 
the multiplicative function given on odd prime powers $\ell^m$ by
\[ \kappa_{-p}(\ell^m) = 
\begin{cases}
{-p \leg \ell} - 1 & m = 1 \\
-{-p \leg \ell} & m = 2 \\
0 & m \ge 3.
\end{cases} \]
\end{theorem}

\begin{remark}\label{dimrem} 
(a) This theorem says that for fixed $N, T, k$ as above, the dimensions of the 
subspaces $S_k^\new(N;\pi_T)$ are of the form $\I + \delta \varepsilon A$ where $\I, A$ 
are constant, and $\delta \in \{ 0, 1 \}$ and $\varepsilon = \pm 1$ depend upon the choice 
of $\pi_p$'s for $p \mid T$.  
The condition $T \ge 5$ odd, which guarantees that there is a just a single elliptic 
orbital integral on the geometric side of the relevant trace formula (see \eqref{dim1}), 
 is necessary for a result of this form.  
For instance, there are four choices for $\pi_3$ in level $27$, and the four spaces 
$S_6^\new(27;\pi_3)$ have dimensions $1, 2, 2, 2$ by \cite[Theorem 7.16]{K};
this is not compatible with the form $\I + \delta \varepsilon A$.  The difference is that 
there are two more elliptic terms in the trace formula in this case.

(b) Some of our results also incorporate unramified (i.e., even conductor exponent) 
  supercuspidal representations in addition to ramified ones.
  Proposition \ref{Tprop} generalizes the $\Delta(\pi_T)=0$ case of the above theorem to allow for 
  prescribed supercuspidals of any conductor exponent,
 giving conditions under which the dimension is just the main term.
In \S\ref{dz} we indicate how one can extend Theorem \ref{dim} to incorporate 
depth zero (conductor $p^2$) supercuspidals at certain places.

(c) The explicit calculations of $\tr(W_T|S_k^{\new}(N))$ that yield \eqref{M1} and \eqref{Tp}
   (see \eqref{WTM1} and \eqref{WTTp}) come from \cite[Proposition 3.2]{M23} 
  and \cite[\S 4.2]{M25} respectively, 
  and are obtained from the (classical) trace formula for Atkin--Lehner operators.  
   (These references contain a couple of typographical errors, 
  which we fix in  the appendix.)  Similar methods should yield an explicit formula for 
  $\tr(W_T|S_k^{\new}(N))$ in general, but we do not attempt to carry this out here.

(d) The factor $\frac{p^2-1}2p^{r_p-1}$ appearing in the main term is the formal degree of $\pi_p$ relative to the
Haar measure for which $\meas(\PGL_2(\Z_p))=1$.  (See Lemma \ref{fdlem} and \S \ref{sec:related} below.)  

(e) The special case of \eqref{M1} in which all $r_p=1$ was first given in \cite[Theorem 1.2]{K}.
As mentioned earlier, a surprising feature of the more general case \eqref{M1} (and also \eqref{Tp})
 is that the elliptic term is the same -- it is unchanged if we increase the conductor exponents of the $\pi_p$'s.
\end{remark}

In \S\ref{scmodel} we give an explicit model for the unitary ramified supercuspidal representations
  of $\GL_2(\Q_p)$, following Kutzko.
In \S\ref{orbits} we show that the Galois orbit of a given $\pi_p$ as in Theorem \ref{dim}
 consists of all supercuspidals of $\PGL_2(\Q_p)$ with conductor $p^{2r_p+1}$ that have 
the same $E_p$ and $\e_p$ as $\pi_p$.  This is used in \S\ref{invariance} to 
prove that $\dim S_k^{\new}(N;\pi_T)$ depends only on the Galois orbits of the local components
of $\pi_T$.  
It follows that the non-archimedean part of the relevant elliptic orbital integral depends only
on $N$, the fields $E_p$, and $\e_{\pi_T}$.
The trace of the Atkin--Lehner operator then provides an additional constraint that determines
the value of this integral, yielding \eqref{main}.   We remark that computing the trace of a Hecke 
  operator $T_n$ on $S_k^{\new}(N;\pi_T)$
  with $n>1$ will generally involve more than one elliptic orbital integral, 
  and so its determination would require more information.

Below we will discuss further context for Theorem \ref{dim}.

\subsection{Relation to root number bias}
For the levels that we consider, Theorem~\ref{dim} identifies more precisely where the local and global root number biases in \cite{M23,M25} arise.
  E.g., if $M = 1$, then we see that the global root number bias 
in \cite{M23} is only coming from the collection of ramified supercuspidals associated to the 
quadratic extensions $E_p/\Q_p$ which make $\Delta(\pi_T) = 1$.
We remark that under certain congruence conditions, one can also deduce this from considering 
the action of quadratic twists on these spaces (see \cite[\S 7]{M25}).

Further, from the perspective of the trace formula, the reason for the bias is simply that
 the local root number appears in the matrix coefficient for $\pi_p$. It factors out
of the relevant local orbital integral (see \eqref{eout} below), leading directly to the root number 
  $\e_{\pi_T}$ in \eqref{main}.

\subsection{Relation to Galois-invariant decompositions}\label{Galdec}

We have been discussing the decomposition of $S_k^\new(N)$ according to all possible local 
  components at $p \mid T$.  
However, for arithmetic investigations it is desirable to decompose $S_k^\new(N)$ according
to Galois orbits of newforms.  Given a Hecke eigenform $f(z)=\sum a_n q^n$ normalized 
  so that $a_n=1$, its Galois orbit is the set of newforms $f^\sigma(z)=\sum \sigma(a_n)q^n$
  for $\sigma\in \Aut(\C)$, or equivalently, $\Gal(\Qb/\Q)$. This action extends $\C$-linearly to
 a Galois action on $S_k^{\new}(N)$.  When $N=1$, Maeda's conjecture
  asserts that there is a single Galois orbit of newforms.

There is no apparent way to detect the Galois orbits of newforms
in $S_k^{\new}(N)$ directly via the trace formula.  
The best one can aim for is to decompose the space according to the Galois orbits 
of local representations $\pi_p$ at each place $p \mid N$.  This leads to a decomposition of
  $S_k^{\new}(N)$ in which each subspace summand is globally Galois-invariant, but not in general minimally so, i.e., each summand may contain multiple Galois orbits.
  However, it is expected that generically each summand is spanned by a single Galois orbit, at least after separating out non-minimal twists and CM forms.
(See, for example, \cite{LS,M21,CM,DPT} for this philosophy, if not this exact statement.)

Suppose $N$ is given by \eqref{N} as above.  For $p \ge 5$, we will show 
  in Proposition~\ref{prop:locgal}
that the $2(p-1)p^{r-1}$ supercuspidal representations of $\PGL_2(\Q_p)$ of conductor $p^{2r+1}$
  are partitioned into exactly four Galois orbits,  parametrized 
  by the pairs $(E_p,\e_p)$ giving the ramified quadratic extension 
  $E_p/\Q_p$ (which specifies the local inertial type) and the 
Atkin--Lehner sign $\e_p$.  One feature of Theorem  \ref{dim} is that the dimension depends
  only on such pairs, i.e., the local Galois orbits of the fixed components at the prime factors of $T$.  
  Thus one can reinterpret 
  the theorem as a formula for the dimension of the subspace of $S_k^\new(N)$ determined 
  by prescribing local Galois orbits for each $p \mid T$.  Namely, each local Galois orbit 
  of conductor $p^{2r+1}$ consists of $\frac{p-1}2 \cdot p^{r-1}$ supercuspidals, so one merely 
needs to multiply the dimension formula in Theorem~\ref{dim} by a product of factors of this form.

We explicate this in the simple case that $N = p^{2r+1}$ and $T=p$.  
\begin{corollary} 
\label{cor}
Let $k \ge 4$ be even, $p \ge 5$, $r \ge 1$, $E_p/\Q_p$ be a ramified quadratic extension, and $\varepsilon_p \in \{ \pm 1 \}$.  Write $S_k^\new(p^{2r+1}; E_p, \varepsilon_p)$ for the (Galois-invariant) subspace of $S_k^\new(p^{2r+1})$ spanned by newforms with local component $\pi_p$ dihedrally induced from $E_p$ with Atkin--Lehner sign $\varepsilon_p$.  Then
\begin{multline*}
\dim S_k^\new(p^{2r+1}; E_p, \varepsilon_p) = \\
\frac{k-1}{12}\left(\frac{p-1}2 \right)^2(p+1)p^{2(r-1)} 
+ {(-1)^{k/2}} \varepsilon_p \, \Delta(E_p) \, b_{p,0} \frac{(p-1) p^{r-1}}2 \, h(-p),
\end{multline*}
where  $\Delta(E_p) = 1$ if $E_p \simeq \Q_p(\sqrt{-p})$ and $0$ otherwise.
\end{corollary}

\begin{remark} One can also deduce the $r=1$ case from \cite[Theorem 7.17]{K}.
\end{remark}

Note that a newform of level $p^{2r+1}$ must have a rationality field which contains 
\begin{equation}\label{Qz+}
\Q(\zeta_{p^r})^+=\Q(\zeta_{p^r}+\zeta_{p^r}^{-1})=\Q(\zeta_{p^r})\cap \R
\end{equation}
 for $\zeta_{p^r}$ a primitive $p^r$-th root of unity 
  (see \cite{M24} or Proposition~\ref{prop:locgal}), so each 
Galois-invariant space $S_k^\new(p^{2r+1}; E_p, \varepsilon_p)$ must have dimension 
a multiple of $\frac 12 \phi(p^r) = \frac{(p-1)p^{r-1}}2$.  This provides a simple sanity 
check on the corollary. 

Corollary~\ref{cor} often allows us to identify the local components $\pi_p$ (up to local Galois conjugacy) for global Galois orbits from the sizes of the Galois orbits together with the Atkin--Lehner signs.  This is considerably simpler than the algorithm presented in \cite{LW}.  (See also \cite{M24} for a partial analogue when $p=3$.)

\begin{example} When $k=4$ and $N = 1331 = 11^3$, Corollary~\ref{cor} says that
$\dim S_4^\new(11^3; E_p, \varepsilon_p)$ is $75$ if $E_p \simeq \Q_{11}(\sqrt{11})$ and $75 + \varepsilon_p 10$ if $E_p \simeq \Q_{11}(\sqrt{-11})$.
One can check in the \cite{LMFDB} that there are six Galois orbits of newforms in $S_4^\new(N)$.  They have sizes $5$, $5$, $60$, $75$, $75$, and $80$ and Atkin--Lehner signs $+1$, $-1$, $-1$, $-1$, $+1$ and $+1$, respectively.  
Necessarily, the two orbits of size $75$ have local components $\pi_{11}$ dihedrally induced 
from $\Q_{11}(\sqrt{11})$ and the other four orbits have $\pi_{11}$ dihedrally induced from 
$\Q_{11}(\sqrt{-11})$.

We remark that the two orbits of size $5$ each consist of CM forms, 
with CM by $\Q(\sqrt{-11})$.  Thus there are ten CM forms in $S_4^\new(11^3)$ and 
$10$ is precisely the size of the secondary term in Corollary~\ref{cor} when $\Delta(E_p) \ne 0$. 
There is a similar numerical coincidence whenever $p \equiv 3 \mod 4$.  So at first 
glance one might wonder whether the secondary term in  Corollary~\ref{cor} can at least 
partially be explained by the existence of CM forms.  However, since the two orbits of CM 
forms occur in spaces with opposite Atkin--Lehner signs, there does not seem to be a direct 
link.
Furthermore, CM forms do not occur in $S_k^\new(p^{2r+1})$ when $p \equiv 1 \mod 4$.  (Such a form would have to have CM by an imaginary quadratic field with discriminant dividing $p^{2r+1}$, but there are no such fields.)
\end{example}

\begin{example} Let $k=4$ and $N=3125 = 5^5$.  Here the newforms have not been computed in the LMFDB, 
but $\dim S^\new_4(5^5) = 600$ and we can compute the Hecke polynomial $h_2$ for $T_2$ acting on 
$S_4^{\new}(5^5)$ in Sage.  Since $h_2$ has distinct irreducible factors of degrees $140$, $150$, $150$ 
and $160$, these must be the sizes of the Galois orbits of newforms.  By Corollary~\ref{cor}, 
the orbit of size $150 \pm 10$ corresponds to a newform $f$ with local component $\pi_5$ 
dihedrally induced from $\Q_5(\sqrt 5) = \Q_5(\sqrt{-5})$ and Atkin--Lehner sign $\pm 1$.
\end{example}

\subsection{Other related work} \label{sec:related}
Several authors before have considered the problem of asymptotics for dimensions of 
newspaces with prescribed local ramified components or inertial types.  See for instance 
\cite{W} for prescribing arbitrary inertial types in the more general setting of Hilbert 
modular forms, and \cite{KST} for prescribing supercuspidal components for more 
general automorphic forms.  This amounts to determining the main term in the trace formula, which involves the formal degree.
Theorem~\ref{dim} shows that, at least in our setting, the exact dimension formula is quite simple,
with the asymptotic being in fact an equality much of the time.
 See also the introduction to \cite{K} for more discussion of such asymptotic formulas.
We discuss inertial types further at the end of \S\ref{Galorb}.

We also remark that the authors of \cite{DPT} considered the problem of existence of cusp forms with
 given components at the ramified places for sufficiently large weight.  For supercuspidal 
components, it is not too hard to deduce this from a simple trace formula.  
One consequence of our exact formula is an effective lower bound for weights where 
all ramified supercuspidals of a given conductor appear.  For instance if $p\ge 5$,
\eqref{M1} implies that all supercuspidals
$\pi_p$ of conductor $2r+1$ occur in $S_k^\new(p^{2r+1})$ for any even $k \ge 4$.  
(One can check it directly for small $p$ and apply the trivial bound $h(-p) < 2p$ for large $p$.)

\subsection*{Acknowledgments}
Support for this research was provided by an AMS-Simons Research Enhancement Grant
for Primarily Undergraduate Institution Faculty, to the first author.

\section{Supercuspidal representations of conductor $\p^{2r+1}$}\label{scmodel}

In this mostly expository section we recall Kutzko's construction of the 
unitary supercuspidal representations of $\GL_2(F)$ with odd-power conductor,  
for $F$ a $p$-adic field.
Any such representation is compactly induced from a character
of an appropriately-chosen open compact-mod-center subgroup.
We follow the description given in \cite[\S1]{Ku} and \cite[\S A.3.8]{H}
 (see also \cite[\S15,\S19]{BH}), making some of the details
more explicit for use later on.

  Let $p$ be a prime number, and let $F$ be a finite extension of $\Q_p$ with
ring of integers $\O$, maximal ideal $\p$, valuation $v$, and $q=|\O/\p|$. 
Fix once and for all a uniformizer $\varpi\in \p$ and a character
\[\psi:F\longrightarrow\C^\times\]
which is nontrivial on $\O$ but trivial on $\p$.
In this section only, we set $G=\GL_2(F)$, and write $Z$ for its center, so $Z\cong F^\times$.
This is also the only section in which we allow for a nontrivial central character $\w$.

Fix an integer $r\ge 1$, and let $n=2r+1\ge 3$.
The central character of a supercuspidal representation of $G$ of conductor $\p^n$ 
 has conductor dividing $\p^r$ (\cite[Prop. 3.4]{T}). 
 Fix such a character 
\[\w:F^\times\longrightarrow\C^\times\]
 trivial on $1+\p^r$.  

\begin{proposition}
For $n=2r+1$ and $\w$ as above, up to isomorphism there are exactly $2q^{r-1}(q-1)$ distinct supercuspidal representations
of $G$ having conductor $\p^n$ and central character $\w$.
\end{proposition}
\begin{proof}
The case of trivial central character is explained in \cite[Theorem 3.9 and its remark]{T}. 
The proof of the general case is actually the same, in view of the following fact: for a finite group $G$
with $Z$ a subgroup of its center, and $\w$ any character of $Z$,
\[
|G/Z| = \sum_{\substack{\pi\in \Irr(G)\\ \w_\pi|_Z=\w}}(\dim \pi)^2,
\]
where $\w_\pi$ denotes the central character of $\pi$.  In the proof and notation of \cite[Cor. 3.6.1]{T},
this can be applied with $G=D^\times/\sg{\varpi}U_D^{i}$ and $Z=F^\times U_D^i/\sg{\varpi}U_D^i$, using
\[
|D^\times/F^\times U_D^i| =2(q+1)q^{2i-\lceil i/2\rceil-1}.\qedhere
\]
\end{proof}

Let $P=\mat \p\O\p\p$ and for $r\ge 1$ define the open compact subgroup
\begin{equation}\label{Ur}
U^r=1+P^r = \mat{1+\p^{s'}}{\p^s}{\p^{s+1}}{1+\p^{s'}},
\end{equation}
for $s=\lfloor \frac r2\rfloor=\lfloor \frac{n-1}4\rfloor$ and
$s'=\lceil \frac r 2\rceil = r-s=\begin{cases}s&\text{if $r$ is even}\\s+1&\text{if $r$ is odd.}\end{cases}\quad$
In \cite{BH}, this group is denoted $U_{\mathfrak A}^r$ for $\mathfrak A = \mathfrak {J}=\smat\O\O\p\O$. 
We have an isomorphism
\[U^{r}/U^{r+1}\longrightarrow (\O/\p)^2
\]
induced by
\begin{equation}\label{Ustep}
\mat{1+a\varpi^{s'}}{b\varpi^s}{c\varpi^{s+1}}{1+d\varpi^{s'}}\mapsto
\begin{cases}
(a,d)\mod \p&\text{if $r$ is even}
\\
 (b,c)\mod \p&\text{if $r$ is odd}.\end{cases}
\end{equation}

Since $\w$ is trivial on $1+\p^r$, it defines a character of 
$(U^r\cap Z)/(1+\p^r)=(1+\p^{s'})/(1+\p^r) \cong \O/\p^s$.  Hence there exists a unique 
\[\alpha=\alpha_\w \in \O/\p^s\]
such that
\begin{equation}\label{wh}
\w(1+\varpi^{s'}d) = \psi\Bigl(\frac{\alpha d}{\varpi^{s-1}}\Bigr)
\end{equation}
for all $d\in\O$.

In Proposition \ref{chiprop} below, we will attach a character
\[\chi=\chi_{t,\w}:U^r\longrightarrow\C^\times\]
to each $t\in\O^\times/(1+\p^{s'})$.
First we establish some notation.  Fix $t\in\O^\times$ and let
\begin{equation}\label{gchi}
g_\chi = \mat{0}t{\varpi}{\varpi \alpha}\in P
\end{equation}
for $\alpha$ as in \eqref{wh}.
The characteristic polynomial 
\begin{equation}\label{Pchi}
X^2-\varpi \alpha X-t\varpi
\end{equation}
of $g_\chi$ is irreducible over $F$ by Eisenstein's criterion, 
so $E=F[g_\chi]$ is a ramified quadratic extension of $F$.  
Notice that $g_\chi\in \O_E$ is a uniformizer since its norm is $\det g_\chi= -t\varpi$.
 Furthermore, by \cite[Prop. I.6.17]{Se}, the ring of integers of $E$ is given by
\[\O_E = \O +\O g_\chi.\]
The maximal ideal of $\O_E$ is
\[\P=\p+\O g_\chi = P\cap E.\] 
Using the fact that $g_\chi^2 = \varpi \alpha g_\chi+\varpi t$, we find by induction that for $\ell \ge0$,
\begin{equation}\label{Pell}
\P^\ell = \p^{\lceil\ell/2\rceil}+\p^{\lfloor \ell/2\rfloor}g_\chi,
\end{equation}
 so in particular
\begin{equation}\label{Pr}
\P^r =\p^{s'}+\p^sg_\chi.
\end{equation}

\begin{proposition}\label{chiprop}
For $t\in \O^\times$ as above and
\begin{equation}\label{k}
k=\mat {1+\varpi^{s'}a}{\varpi^sb}{\varpi^{s+1} c}{1+\varpi^{s'}d}\in U^r,\end{equation}
 define
\begin{equation}\label{chi}
\chi(k) = \psi\Bigl(\frac{\tr(g_\chi (k-1))}{\varpi^{r}}\Bigr)=
\psi\Bigl(\frac{\varpi^{s+1}(b+tc)+\varpi^{s'+1}\alpha d}{\varpi^{r}}\Bigr)=
\w(1+\varpi^{s'}d)\psi\Bigl(\frac{b+tc}{\varpi^{s'-1}}\Bigr).
\end{equation}
Then $\chi$ is a character of $U^r$ depending only on $t\mod 1+\p^{s'}$, with
\begin{equation}\label{kerchi}
U^{2r}\subsetneq \ker\chi.
\end{equation}
Furthermore, $\chi$ extends to a character of $Z U^r$ via $\chi|_{Z}=\w$.
Lastly, the element $g_\chi\in G(F)$ normalizes $U^r$ and
\begin{equation}\label{gkg}
\chi(g_\chi^{-1}xg_\chi)=\chi(x)
\end{equation}
for all $x\in ZU^r$. 
\end{proposition}
\begin{remark}
The group $U^{2r-1}=\mat{1+\p^r}{\p^{r-1}}{\p^r}{1+\p^r}$ contains $U^{2r}$ but
not $\ker\chi$. 
For example, the matrix $\mat1{\varpi^s\varpi^{s'-2}t}{-\varpi^{s+1}\varpi^{s'-2}}1\in U^r$ is not in $U^{2r-1}$ but
it belongs to $\ker\chi$.
\end{remark}

\begin{proof}
First we check that $\chi$ is a homomorphism.
For $k'=\mat {1+\varpi^{s'}a'}{\varpi^sb'}{\varpi^{s+1} c'}{1+\varpi^{s'}d'}\in U^r$,
\[kk'= \mat{1+\varpi^{s'}(a+a')+\varpi^{2s'}aa'+\varpi^{2s+1}bc'}
 {\varpi^s(b+b')+\varpi^r(a'b+b'd)}
{\varpi^{s+1}(c+c')+\varpi^{r+1}(a'c+dc')} {1+\varpi^{s'}(d+d')+\varpi^{2s'}dd'+\varpi^{2s+1}b'c}.\]
It follows that
\[\chi(kk')= \psi\Bigl(\frac{(b+b')+t(c+c')}{\varpi^{s'-1}}+\frac{\alpha(d+d')}{\varpi^{s-1}}\Bigr)
=\chi(k)\chi(k')\]
for $\alpha$ as in \eqref{wh}, as required.

Noting that
\[\mat{1+\p^{s'}}{\p^r}{\p^{r+1}}{1+\p^{r-v_\p(\alpha)}}\subset\ker\chi,\]
where $0\le v_\p(\alpha)\le s$, \eqref{kerchi} follows.

By \eqref{chi} (whose third equality comes from \eqref{wh}), $\chi(z)=\w(z)$ for $z\in Z\cap U^r$. 
We can therefore extend $\chi$ to a character of $ZU^r$.

Note that
\[g_\chi^{-1}P g_\chi = \mat{-\alpha/t}{1/\varpi}{1/t}0\mat\p\O\p\p\mat0t\varpi{\varpi \alpha}
=\mat\O\O\p\O\mat0t\varpi{\varpi \alpha}=P.\]
Consequently, $g_\chi$ normalizes $U^r=1+P^r$. Furthermore, for $k\in U^r$,
\[
\chi(g_\chi^{-1}kg_\chi)=\psi\Bigl(\frac{\tr((k-1)g_\chi)}{\varpi^{r}}\Bigr)
=\psi\Bigl(\frac{\tr(g_\chi(k-1))}{\varpi^{r}}\Bigr) = \chi(k),
\]
giving \eqref{gkg}.
\end{proof}

Henceforth we will view $\chi=\chi_{t,\w}$ as a character of $ZU^r$ as in the proposition.
Following \cite[Def. 1.5]{Ku}, let $\Lambda_{\chi}=\Lambda_{t,\w}$ be the set of characters
$\lambda$ of $E^\times$ that satisfy $\lambda|_{F^\times}=\w$ and whose restrictions to 
\begin{equation}\label{UEr}
E^\times\cap U^r=1+\p^{s'}+\p^sg_\chi=1+\P^r=:U_E^r
\end{equation}
 (see \eqref{Pr}) coincide with the restriction
of $\chi$ to this set.
In the simplest case where $n=2r+1=3$, $\Lambda_\chi=\{\chi\}$ is a singleton set.

Recall that the {\em level} of $\lambda$ is the smallest integer $k\ge 0$ such that
$\lambda$ is trivial on $U_E^{k+1}:=1+\P^{k+1}$.

\begin{lemma}\label{level}
Let $\lambda\in \Lambda_\chi$.   If $\P$ is odd, then $\lambda$ has level $n-2=2r-1$ and 
$\lambda$ determines $t$, and hence $\chi$.
If $\P\mid 2$, then the level of $\lambda$ is $\le n-3$ and $\lambda$ does not determine $\chi$.
\end{lemma}

\begin{proof}
By \eqref{Pell},
\[\P^{n-2}=\P^{2r-1}=\p^r+\p^{r-1}g_\chi.\]
Since $r\ge 1$, $n-2=2r-1\ge r$, so $1+\P^{n-2}\subset 1+\P^r\subset U^r$.  Thus for $a,c\in\O$, 
\[\lambda(1+a\varpi^r+c\varpi^{r-1}g_\chi)=
\chi(\mat{1+a\varpi^r}{ct\varpi^{r-1}}{c\varpi^r}{1+a\varpi^r+c\varpi^r\alpha})
=\psi(2tc)\]
using the fact that $\w$ is trivial on $1+\p^r$.
If $\P$ is odd, this is a nontrivial function of $c$, so $\lambda$ is nontrivial on $1+\P^{n-2}$. 
On the other hand, by \eqref{kerchi}, $\lambda$ is trivial on $U_E^{n-1}\subset U^{n-1}=U^{2r}$. 
Thus $\lambda$ has level $n-2$.  
If $\P\mid 2$, then $\psi(2tc)=1$ and so $\lambda$ is trivial on $U_E^{n-2}$.

Similarly, for any $b\in \O$, $1+b\varpi^s g_\chi\in U_E^r$ by
\eqref{UEr}, and
\[\lambda(1+b\varpi^sg_\chi) = \chi(\mat{1}{bt\varpi^s}{b\varpi^s+1}{1+b\varpi^{s+1}\alpha})
  =\w(1+b\varpi^{s+1}\alpha)\psi(\frac{2bt}{\varpi^{s'-1}}).\]
Thus, given the fixed central character $\w$, $\lambda$ determines $t\in \O^\times/(1+\p^{s'})$ when
$\p$ is odd, but $t$ is only determined modulo $1+\p^{s'-v_\p(2)}$ if $\p\mid 2$.
\end{proof}

Fix $\lambda\in \Lambda_{\chi}$ and consider the restriction $\lambda|_{\O_E^\times}$.
By \eqref{kerchi}, it may be viewed
as a character of the finite group
\begin{equation}\label{lgroup}
\O_E^\times/U_E^{n-1} \cong \mu_{q-1}\times (U_E^1/U_E^{n-1}),
\end{equation}
where $\mu_{q-1}\subset \O^\times$ consists of the $(q-1)$-st roots of unity.
(Since $E/F$ is ramified, they both have the same residue degree $q$.)
An explicit parametrization of $\Lambda_{\chi,\w}$ could be given using the structure of the above
abelian group, given in \cite{nakagoshi}.

In general, if $G$ is a finite abelian group with a subgroup $H$, then restricting
characters gives a surjective homomorphism 
\[\Res:\widehat{G}\longrightarrow\widehat{H}\]
of the dual groups.  Thus each character $\chi\in \widehat{H}$ has exactly $|G/H|$ distinct
 extensions to $G$.
In our situation, the given character $\chi$ (restricted to $F^\times U_E^r$) has
\[|\O_E^\times/\O^\times U_E^r|\]
extensions to $\O_E^\times$. 
Noting that $\P^r\cap \O = \p^{s'}$ as in \eqref{Pr}, we find
\begin{equation}\label{OEindex}
|\O_E^\times/\O^\times U_E^r|=\frac{|\O_E/\P^{r-1}|}{|\O/\p^{s'-1}|} = q^{r-1}/q^{s'-1}=q^s.
\end{equation}

Finally, given an extension $\lambda$ of $\chi$ to $\O_E^\times$ as above, 
it can be extended to $E^\times$ by defining it on the prime element $g_\chi$.
In view of \eqref{Pchi}, we must have
\[\lambda(g_\chi)^2 = \lambda(\varpi t+\varpi \alpha g_\chi)=\w(\varpi)\lambda(t+\alpha g_\chi).\]
Both factors on the right-hand side are defined, since $t+\alpha g_\chi\in \O_E^\times$.
There are thus two choices for $\lambda(g_\chi)\in\C$.
This proves the following.

\begin{proposition}\label{Lw}
  Having fixed $\w$, there are $q^{s'-1}(q-1)$ characters $\chi$ as in
\eqref{chi}, corresponding to the set of $t\in \O^\times/(1+\p^{s'})$.  For each such $\chi$,
\[|\Lambda_{\chi}|=2q^s.\]
Consequently, $\Lambda_{\w,r} := \bigcup_\chi \Lambda_{\chi}=\bigcup_t\Lambda_{t,\w}$ 
  has $2q^{r-1}(q-1)$ elements.
\end{proposition}

Now fix $\chi$ and define
\begin{equation}\label{Jn}
J_{E,r} = E^\times U^r. 
\end{equation}
It is an open subgroup of $G$ containing, and compact modulo, $Z$.
For $\lambda\in \Lambda_{\chi}$, we may extend
$\chi$ to a character of $J_{E,r}$ by
\[\chi_\lambda(x k) = \lambda(x)\chi(k)\]
for $x\in E^\times$ and $k\in U^r$. 
We then define the compactly induced representation
\[\pi_{\chi_\lambda} =\cInd_{J_{E,r}}^G(\chi_\lambda).\]
In view of the fact (Lemma \ref{level}) that $\lambda$ determines $\chi$ when $\p$ is odd,
in such cases, we can write $\pi_\lambda$ instead of $\pi_{\chi_\lambda}$.

\begin{proposition}\label{2main}
For $\chi_\lambda$ as above,
$\pi_{\chi_\lambda}$ (or simply $\pi_\lambda$ if $\p$ is odd) is irreducible and 
  supercuspidal of conductor $\p^n$, where $n=2r+1$.
The $2q^{r-1}(q-1)$ representations $\pi_{\chi_\lambda}$
thus obtained are mutually inequivalent, so they comprise the set of all supercuspidals 
of conductor $\p^n$ and central character $\w$.
The new vector of $\pi_{\chi_\lambda}$ is supported on the double coset
\begin{equation}\label{JpK}
J_{E,r} \mat{\varpi^r}{}{}1 K_1(\p^n),\end{equation}
where $K_1(\p^n)=\mat{\O^\times}\O{\p^n}{1+\p^n}$.
When $\w$ is trivial, the root number of $\pi_{\chi_\lambda}$ is 
\begin{equation}\label{epi}
\e=\lambda(g_\chi)\in \{\pm 1\}.
\end{equation}
\end{proposition}
\begin{remark}  In the notation of \cite[\S15]{BH}, $\pi_{\chi_\lambda}$ is the 
representation attached to the cuspidal type $(\mathcal{J}, 2r-1, \varpi^{-r}g_\chi)$. 
\end{remark}

\begin{proof}
Irreducibility and supercuspidality are proven in \cite[Prop. 1.7]{Ku}, with inequivalence proven
in \cite[Prop. 2.9]{Ku}. See also \cite[\S15]{BH}.

We will verify momentarily that $\pi_{\chi_\lambda}$ has a $K_1(\p^{2r+1})$-fixed vector, so that
the conductor divides $\p^{2r+1}$.
Using the fact \cite[Prop. 3.5]{T} that $E/F$ is ramified if and only if the conductor exponent is odd,
along with the count of supercuspidals of a given conductor and central character, 
it follows by induction on $r$ that
the conductor of $\pi_{\chi_\lambda}$ is exactly $\p^{2r+1}$.

  In order to show that $\pi_{\chi_\lambda}$ contains a well-defined $K_1(\p^n)$-invariant
function on the double coset \eqref{JpK}, we need to show that 
$\chi_\lambda(h_1)= \chi_\lambda(h_2)$ whenever
\begin{equation}\label{hi}
h_1\smat{\varpi^r}{}{}1 g_1 = h_2\smat{\varpi^r}{}{}1g_2
\end{equation}
for some $h_1,h_2\in J_{E,r}$ and $g_1,g_2\in K_1(\p^n)$.
Write $h_i=g_\chi^{d_i}z_ik_i$ for $d_i\in\{0,1\}$, $k_i\in U^r$, and $z_i\in Z$.
  From the valuation of the determinant
in \eqref{hi} we conclude that $d_1=d_2$ and $z_2^{-1}z_1\in \O^\times$. 
  So without loss of generality we can assume that $h_i=z_ik_i$ with $z_i\in\O^\times$.
We may then write
\[k:=z_2^{-1}z_1k_2^{-1}k_1 = \mat{\varpi^r}{}{}1\mat ab{\varpi^n c}{1+\varpi^n d}\mat{\varpi^{-r}}{}{}1
=\mat a{b\varpi^r}
{c\varpi^{r+1}}{1+\varpi^n d}\]
where $\smat ab{\varpi^n c}{1+\varpi^n d}=g_2g_1^{-1}\in K_1(\p^n)$.
As $k_2^{-1}k_1\in U^r$, the lower right entry of $k$ belongs to $z_2^{-1}z_1+\p^{s'}$.  But this entry 
also equals $1+\varpi^n d$. It follows that $z_2^{-1}z_1$, and hence also $k$,
 belongs to $U^r$.  Therefore we can evaluate $\chi_\lambda(k)=\chi(k)$ using \eqref{chi}, giving
\[\chi(k)= \psi\Bigl(\frac{b\varpi^{s'}+tc\varpi^{s'}}{\varpi^{s'-1}}
  +\frac{\alpha\varpi^{n-s'}d}{\varpi^{s-1}}\Bigr) =1,\]
as required.

Now assume $\w$ is trivial,  so $\alpha=0$ and $g_\chi=\smat{}t\varpi{}$. 
Let $\phi$ be the newvector of $\pi=\pi_{\chi_\lambda}$
 satisfying $\phi(\smat{\varpi^r}{}{}1)=1$.  Then
\[\pi(\mat{}1{\varpi^n}{})\phi = \e\phi\]
for the root number $\e$ of $\pi$, \cite[Thm 3.2.2]{Sch}.  So 
\begin{align*}
\e&= \Bigl[\pi(\mat{}1{\varpi^n}{})\phi\Bigr](\mat{\varpi^r}{}{}1) 
= \phi(\mat{}{\varpi^r}{\varpi^n}{})\\
&=\phi(\mat{}t{\varpi}{}\mat{\varpi^r/t}{}{}{\varpi^r/t}\mat{\varpi^r}{}{}1\mat t{}{}1)
=\chi_\lambda(g_\chi)=\lambda(g_\chi).
\end{align*}
Note that $\lambda(g_\chi)^2=\lambda(t\varpi) = \w(t\varpi)=1$ since $\w$ is trivial.
\end{proof}

\section{Local Galois orbits}\label{orbits}

We continue the local setup and notation of the previous section.  In particular, $F$ is a $p$-adic field.

\subsection{Galois action}\label{Galact}

The automorphism group of $\C$ acts on complex representations of a group by automorphisms of 
the coefficients.  This action is given in detail as follows.
For $V$ a complex vector space and $\sigma\in \Aut(\C)$, let $V^\sigma$ denote
the vector space whose underlying group is $V$, but with scalar multiplication given by 
$a\cdot v= \sigma^{-1}(a)v$.  If $G$ is a group and $\pi:G\to\GL(V)$ is a representation,
we let $\pi^\sigma$ denote the representation of $G$ on $V^\sigma$ defined by 
  $\pi^\sigma(g)\cdot v=\pi(g)v$.
We say that a representation $\pi' : G \to \GL(V')$ is in the \emph{Galois orbit} 
of $\pi$ if $\pi' \simeq \pi^\sigma$ for some $\sigma \in  \Aut(\C)$.

If $\sg{v,w}$ is the canonical bilinear pairing on $V\times V^*$, then 
$(V^\sigma)^*$ may be identified as a set
with $V^*$, with the pairing on $V^\sigma\times (V^\sigma)^*$ given by 
\[\sg{v,w}_\sigma := \sigma(\sg{v,w}).\]
For example, $\sg{\lambda\cdot v,w}_\sigma = \sigma(\sg{\sigma^{-1}(\lambda)v,w})
  =\lambda\sg{v,w}_\sigma$.
Furthermore, if $\phi(g)=\sg{\pi(g)v,w}$ is a matrix
coefficient for $\pi$, then $\sigma(\phi(g))=\sg{\pi^\sigma(g)v,w}_\sigma$ is the corresponding
matrix coefficient for $\pi^\sigma$.
In particular, if $V=\C$ and $\chi$ is a character of $G$, then
$\chi^\sigma(g)=\sigma(\chi(g))$. 

If $\pi=\Ind_H^G(\tau)$ for a representation $(\tau,W)$ of a subgroup $H$ of $G$, then it follows
immediately from the definitions that 
\begin{equation}\label{GaloisInd}
\pi^\sigma=\Ind_H^G(\tau^\sigma).
\end{equation}
 One corollary of this observation is the following.

\begin{proposition}\label{GaloisSatake}
Suppose $\pi$ is an unramified principal series representation of $\PGL_2(F)$ with Satake 
parameters $\{\alpha,\alpha^{-1}\}$.  Then $\pi^\sigma$ is the unramified representation with
Satake parameters $\{\sigma(\alpha),\sigma(\alpha)^{-1}\}$.
\end{proposition}

\subsection{Galois orbit of a ramified supercuspidal}\label{Galorb}

Our goal here is to determine the Galois orbit of a ramified supercuspidal representation.
A direct proof is possible using arguments similar to what appears below, but it is a bit easier 
to work on the Galois side of the tame local Langlands correspondence, 
which we now recall (see \cite[\S34]{BH}).

Throughout this section, $F$ is a finite extension of $\Q_p$ for a prime $p \neq 2$.
Let $E/F$ be a quadratic extension, and $\xi$ an admissible character of $E^\times$.  
(This means that $\xi$ does not factor through the norm map $N_{E/F}$, and if $E/F$ is 
ramified $\xi |_{U^1_E}$ also does not factor through the norm map.) 
Via class field theory, $\xi$ can be viewed as a character of the Weil group $W_E$, 
and its induction 
\[\rho_{\xi}=\Ind_{W_E}^{W_F}(\xi)\]
is a smooth irreducible 2-dimensional representation of $W_F$.
By \cite[\S29.2]{BH},  
\[\det \rho_\xi = \eta_{E/F} \xi|_{F^\times}\]
where $\eta_{E/F}$ is the quadratic character of $F^\times$ associated to $E/F$ by class field theory.
The tame local 
Langlands correspondence associates to $\rho_\xi$ the dihedral supercuspidal representation 
$\pi(\rho_\xi) := \pi_{\lambda}$,
where $\lambda = \Delta_\xi \xi$ and $\Delta_\xi$ is the character of $E^\times/U_E^1$ 
associated to $(E,\xi)$ as in \cite[\S 34.4]{BH}.  In particular, 
$\Delta_\xi |_{F^\times} = \eta_{E/F}$, and if $E/F$ is unramified then
$\Delta_\xi$ is unramified quadratic.
Since $p\neq 2$, every supercuspidal representation of $\GL_2(F)$ is obtained in this way.

We will focus on the case of trivial central character, which means $\xi|_{F^\times}=\eta_{E/F}$.
Consequently, if $\ol \xi$ denotes the $\Gal(E/F)$-conjugate of $\xi$, then 
$\xi(x)\ol\xi(x) =\xi(N_{E/F}(x))=1$, so $\ol\xi=\xi^{-1}$.  Note that
\[\rho_\xi |_{W_E} \simeq \xi \oplus \ol \xi.\]

 Via \S\ref{Galact}, $ \Aut(\C)$ acts on the set of (isomorphism classes of) $n$-dimensional 
Weil (or Weil--Deligne) representations, as well as on representations of $\GL_n(F)$.  
Furthermore, the local Langlands correspondence commutes with this Galois action \cite[\S 7.4]{H2001}.
Applying \eqref{GaloisInd} to $\rho_\xi$, it then follows immediately that
\begin{equation}\label{pisigma}
\pi^\sigma = \pi(\rho_{\xi^\sigma}).
\end{equation}
In addition, from the definition of the Galois action on representations of $\GL_n(F)$ we see that it preserves conductors.

The equality \eqref{pisigma} gives one description of the Galois orbit of a dihedral 
supercuspidal representation.  
We will require a more explicit description in the case of a
ramified supercuspidal representation $\pi = \pi(\rho)$ (see Proposition \ref{prop:locgal}).
 In this case, there is a 
unique ramified quadratic extension $E = E_\pi$ of $F$ such that $\rho$ is induced 
from $E$ (see the proof of \cite[Theorem 34.1]{BH})
and thus by \eqref{pisigma} $E$ is an invariant for the Galois orbit of $\pi$.

Say $\pi = \pi(\rho_\xi)$ is a supercuspidal representation of $\PGL_2(F)$ of conductor exponent $2r+1$.  
Then $\xi : W_E^{\text{ab}} \simeq E^\times \to \C^\times$ 
is a character of conductor exponent $2r$ (\cite[Theorem 2.3.3]{Sch})
such that $\xi |_{F^\times} = \eta_{E/F}$.  
Choose uniformizers $\varpi_F, \varpi_E$ of $F$ and $E$ so that $\varpi_E^2 = \varpi_F$.  
Then $\xi(\varpi_E) = \pm 1$ since $\xi(\varpi_F) = 1$.
Note that $\xi(\varpi_E)$ is closely related to the root number of $\pi$ 
since we can take $\varpi_E=g_\chi$ due to the assumption of trivial central character, 
so by \eqref{epi}, 
  \[\e_\pi=\lambda(\varpi_E)=\xi(\varpi_E)\Delta_\xi(\varpi_E).\]
Because $\xi^\sigma(\varpi_E) = \xi(\varpi_E)\in\Q$ for any 
  $\sigma \in  \Aut(\C)$, the Galois orbit of $\pi$ is determined by 
(i) $E$, (ii) $\xi(\varpi_E)$, and (iii) the Galois orbit of $\xi|_{\O_E^\times}$.

Let $\xi' = (\Delta_\xi \xi) |_{\O_E^\times}$.  
Since $E/F$ is ramified, $\O_E^\times = \O_F^\times U_E^1$, so $\Delta_\xi|_{\O_E^\times}$ 
factors through $\O_F^\times/(\O_F^\times\cap U_E^1)$ on which it agrees with $\eta_{E/F}$.
In particular, this restriction is quadratic and only depends on $E$.  
Thus 
\[(\xi')^\sigma = \Delta_\xi^\sigma \xi^\sigma |_{\O_E^\times} = \Delta_\xi \xi^\sigma |_{\O_E^\times} 
= (\xi^\sigma)'.\]
  So it is sufficient to study the Galois orbit of $\xi'$.
The advantage of working with $\xi'$ is that it factors through $\O_E^\times/\O_F^\times U_E^{2r}$
 (see \eqref{lgroup}).  

\begin{lemma}
Suppose $E/F$ is ramified.  Then
a character of $\O_E^\times/\O_F^\times U_E^{2r}$ is nontrivial on $U_E^{2r-1}$ if and only if it
is of the form $\xi'$ for an admissible character $\xi$ of conductor $2r$ such that
$\xi |_{F^\times} = \eta_{E/F}$.
\end{lemma}

\begin{proof}
Given $\xi$ as above,  $\xi|_{U_E^{2r-1}}=\xi'|_{U_E^{2r-1}}$
 is nontrivial by Lemma \ref{level}.  

Conversely, suppose $\eta$ is a character of $\O_E^\times/\O_F^\times U_E^{2r}$.
By \cite[Cor. V.3.3]{Se} (for which in the present context we have $t=0$ and $\psi(n)=2n$),
\[N_{E/F}(U_E^{2r-1}) = N_{E/F}(U_E^{2r}) = U_F^r\subset U_E^{2r}\subset \ker\eta.\]
If $\eta$ is nontrivial on $U_E^{2r-1}$, it follows that $\eta$ does not factor through the 
norm map.   So it comes from an admissible character $\xi$.
\end{proof}

Consequently, understanding the Galois orbits of ramified supercuspidals requires
understanding the structure of
$\O_E^\times/\O_F^\times U_E^{2r} \simeq (\O_E^\times/U_E^{2r} )/(\O_F^\times/U_F^r)$.  
This group has order $q^r$ (see \eqref{OEindex}).  Its structure can be understood from that 
  of unit groups mod higher unit groups, as determined in \cite{nakagoshi}.  
  However, the structure of the latter is a bit technical and 
breaks up into several cases.  Essentially, the difference in the $p$-ranks of 
$\O_E^\times/U_E^{2r}$ and $\O_F^\times/U_F^r$ is typically $[F:\Q_p]$ when $r \gg_F 0$, so generally this quotient is not cyclic.  
For simplicity, we will just determine certain hypotheses under which $\O_E^\times/\O_F^\times U_E^{2r}$ 
{\em is} cyclic.

\begin{lemma} \label{lem:cyclic} Suppose $E/F$ is a ramified quadratic extension and $r\ge 1$.

\begin{enumerate}
\item If $\O_E^\times/\O_F^\times U_E^{2r}$ is cyclic, then $F/\Q_p$ is totally ramified (including
  the possibility $F=\Q_p$).  
If $F/\Q_p$ is totally ramified, then taking $r=1$, $\O_E^\times/\O_F^\times U_E^{2}$ is 
  cyclic of order $p$.

\item If $F = \Q_p$ and $p \ge 5$,  then $\O_E^\times/\O_F^\times U_E^{2r}$ 
is cyclic of order $p^r$.
\end{enumerate}
\end{lemma}
\begin{remark} When $F=\Q_3$ and $r\ge 2$, the quotient is not cyclic, \cite[\S13]{ranum}.
\end{remark}

\begin{proof}
(1) Since $\O_E^\times/\O_F^\times U_E^{2}$ is a quotient of 
$\O_E^\times/\O_F^\times U_E^{2r}$,  if the latter group is cyclic, so is the former.  
Note that 
\[\O_E^\times/\O_F^\times U_E^{2} \simeq (\O_E^\times/U_E^2)/(\O_F^\times/U_F^1) \simeq U_E^1/U_E^2,\]
 which is isomorphic to the additive group of $\F_q$.  This is only cyclic when $q=p$, 
i.e., $F/\Q_p$ is totally ramified.

(2) By (1), we may assume $r \ge 2$.  Say $E = \Q_p(\sqrt d)$ with $d$ a squarefree integer and 
$p \ge 5$.  Generators and relations for $\O_E^\times/U_E^{2r}$ are determined in \cite[\S 13]{ranum}.
  It is isomorphic to the product of $\F_p^\times$ and a $p$-group of $p$-rank 2.  The 
$\F_p^\times$ factor is generated by an element of $\O_F^\times$, and the $p$-part is 
generated by $1+p \in \O_F^\times$ and $1+\sqrt d$.  Hence $\O_E^\times/\O_F^\times U_E^{2r}$ is 
generated by the single element $1 + \sqrt d$, which has order $p^r$.
\end{proof}

Suppose $\O_E^\times/\O_F^\times U_E^{2r}$ is cyclic, necessarily of order $q^r=p^r$. 
The primitive characters $\xi$ of $\O_E^\times/\O_F^\times U_E^{2r}$ are those which are 
nontrivial on $U_E^{2r-1}$.  Since $\O_E^\times/\O_F^\times U_E^{2r-1}$ has order $p^{r-1}$, the 
imprimitive characters of $\O_E^\times/\O_F^\times U_E^{2r}$ are those with order dividing $p^{r-1}$.  
Hence the primitive characters $\xi$ are in (non-canonical) bijection with the $p^r$-th roots 
of unity which are not $p^{r-1}$-th roots of unity, simply by specifying $\xi$ on a fixed 
generator.  Two such characters are Galois conjugate if and only if they have the same order.  
We see that there are $(p-1)p^{r-1}$ primitive characters, forming a single Galois orbit.

\begin{proposition} \label{prop:locgal}
Suppose $p$ is odd and $r \ge 1$.  Assume that either (i) $F/\Q_p$ is totally ramified
 (this includes the case $F=\Q_p$) and $r=1$, or (ii) $F = \Q_p$ for $p\ge 5$.  
Then there are precisely four Galois orbits of smooth irreducible representations $\pi$ of $\PGL_2(F)$ 
of conductor exponent $2r+1$, and a complete set of invariants for the Galois orbit of $\pi$ is the 
pair $(E_\pi, \varepsilon_\pi)$, where $E_\pi$ is the ramified quadratic extension of $F$ attached
 to $\pi$ and $\varepsilon_\pi$ is the root number of $\pi$. 
Moreover, with notation as in \eqref{Qz+}, $\Gal(\Q(\zeta_{p^r})^+/\Q)$ acts transitively 
and faithfully on the Galois orbit of such a $\pi$.
\end{proposition}

\begin{remark} 
The case of $F=\Q_p$ is also implicit in the proof of \cite[Theorem 2.7]{DPT}, 
though the argument there is somewhat different.
\end{remark}

\begin{proof}
Let $\pi=\pi(\rho_\xi)$ and $\tilde\pi=\pi(\rho_{\tilde\xi})$ be supercuspidal representations 
of $\PGL_2(F)$ of conductor exponent $2r+1$.  Suppose $E_{\pi}=E_{\tilde\pi}=E$ and 
$\xi(\varpi_E)=\tilde\xi(\varpi_E)$, where $\varpi_E^2$ is a uniformizer of $F$.
Under the given hypotheses, $f=1$, so by the preceding discussion there exists
  $\sigma\in  \Aut(\C)$ such that
$\tilde\xi|_{\O_E^\times}=\xi^\sigma|_{\O_E^\times}$. 
But then $\tilde\xi=\xi^\sigma$ on $E^\times$ as well, so $\tilde\pi=\pi^\sigma$.  Since there are 
exactly two possibilities for $E$ and two possibilities for $\xi(\varpi_E)\in\{\pm 1\}$, 
it follows that there are exactly four possibilities for the Galois orbit of $\pi$.
Further, using the fact that the local root number $\varepsilon_\pi$ is 
the eigenvalue of the Atkin--Lehner operator $W$ on a newvector in $\pi$, it is easy 
to see that $\varepsilon_\pi$ is also an invariant of the Galois orbit of $\pi$ that can be used
in place of the related parameter $\xi(\varpi_E)$.

It remains to prove the last statement.  A given $\pi$ as above is parametrized (in a
Galois-equivariant way) by the pair $\{\xi,\xi^{-1}\}$ since $\xi$ and $\xi^{-1}$ induce
isomorphic supercuspidals.   By the preceding discussion, $\xi$ is determined by 
(i) $\xi(\varpi_E)\in\{\pm 1\}$ and (ii) the primitive $p^r$-th root of unity $\xi(\beta)$ where $\beta$
is a generator of the cyclic group $\O_E^\times/\O_F^\times U_E^{2r}$.
  Since $\Gal(\Q(\zeta_{p^r})^+/\Q)$ acts faithfully and transitively on the collection 
of 2-element sets $\{ \zeta, \zeta^{-1} \}$ as $\zeta$ ranges over primitive $p^r$-th roots of 
unity, it also acts faithfully and transitively on the Galois orbit of $\pi$.
\end{proof}

\begin{remark} 
In general, there are more than four Galois orbits of supercuspidal representations of $\PGL_2(F)$
having conductor exponent $2r+1$.  For instance, if $q=p^f > p$, while the values of 
a character $\xi$ can be chosen independently on different generators $a_1, \dots, a_t$
of the noncyclic group $\O_E^\times/\O_F^\times U_E^{2r}$,
Galois automorphisms do not behave independently on the elements $\xi(a_1), \dots, \xi(a_t)$.
\end{remark}

Finally we relate the above discussion to inertial types.   We say that two smooth irreducible 
representations $\pi, \pi'$ of $\GL_2(F)$ are in the same (inertial) type if their 
Langlands parameters $\rho, \rho'$ have isomorphic restrictions to the inertia group 
$I_F \subset W_F$.  If $\pi$ is supercuspidal, then $\pi'$ is in the same type as $\pi$ 
if and only if $\pi' \simeq \pi \otimes \chi$, where $\chi$ is an unramified character of 
$F^\times$.  See \cite{H} for more details.

Let $\pi = \pi(\rho_\xi)$ be a supercuspidal of $\PGL_2(F)$ of conductor $2r+1$, and 
$\chi$ be the unramified quadratic character of $F^\times$.  Then 
$\pi \otimes \chi = \pi(\rho_\xi \otimes \chi)$ is the only other representation of $\PGL_2(F)$ with the same 
inertial type as $\pi$.  The Langlands parameters for both $\pi$ and 
$\pi \otimes \chi$ have the same restriction $\xi|_{\O_E^\times}$ to the inertia group 
$I_E$ of $E$.  Further, $\pi$ and $\pi \otimes \chi$ have opposite 
root numbers (e.g., see \cite[(5.5)]{deligne}).  Hence under the hypotheses of Proposition~\ref{prop:locgal}, we see that 
the Galois orbit of the inertial type of $\pi$ (among representations with trivial central 
character) consists of all supercuspidals $\pi'$ of the same conductor such that $E_\pi = E_{\pi'}$.

\begin{remark} 
In \cite{DPT}, the authors studied Galois orbits of inertial types.  
In fact, because the inertial type is not a fine enough invariant for their end goal, 
\cite{DPT} augment inertial type Galois orbits with ``minimal'' Atkin--Lehner signs.  
We suggest that alternatively one might consider Galois orbits of representations as above, rather 
than Galois orbits of types together with a sign.

When $p = 3$ and $r=1$,
  there are only four supercuspidal representations of $\PGL_2(\Q_3)$ of conductor $3^3$.
 Case (i) of the above proposition implies that the Galois orbits of these representations 
  are singleton sets.  
(It is asserted in \cite[Theorem 2.7]{DPT} that there are four Galois orbits of 
local inertial types in this situation. But this cannot be true since it would imply that there are 
eight Galois orbits of supercuspidals of conductor $3^3$.)
\end{remark}

\section{Dimension formulas}

\subsection{General setup}\label{setup}

Let $G=\GL_2$, $Z$ its center, and $\olG=G/Z$.  Let $\A$ be the adele ring of $\Q$, and normalize Haar measure on $\olG(\A)$ so that
\[\meas(\olG(\Q)\bs\olG(\A))=\pi/3\]
 with $\olG(\Q)$ having the counting measure. 
Write $K_p = \GL_2(\Z_p)$, $\ol{K}_p$ its image in $\olG(\Q_p)$, and normalize the Haar measure on 
 $\olG(\Q_p)$ so that $\meas(\ol{K}_p)=1$.  There is a unique choice of Haar measure
on $\olG(\R)$ so that the global measure on $\olG(\A)$ fixed above is the restricted product of 
the local measures.

We will need an explicit expression for the formal degree of a supercuspidal representation.
The following is a special case of a result of Carayol.

\begin{lemma}\label{fdlem}
Let $\sigma$ be a supercuspidal representation of $\GL_2(F)$ of conductor $q^c$.  Then
the formal degree of $\sigma$, computed relative to the Haar measure normalized by 
$\meas(\ol K)=1$, is 
\begin{equation}\label{fd}
d_\sigma = \begin{cases} \frac{q^2-1}2q^{r-1}&\text{if $c=2r+1$ is odd}\\
  (q-1)q^{r-1}&\text{if $c=2r$ is even}.\end{cases}
\end{equation}
In particular, having fixed the measure, the formal degree depends only on the conductor of $\sigma$.
\end{lemma}
\begin{proof}
In \cite[\S5.11]{C}, Carayol computed the formal degrees of the supercuspidal representations
of $\GL_n(F)$.  Using the measure $\meas_C$ normalized by 
  $\meas_C(\ol{K_0(\p)})=n^{-1}(q^n-1)^{-1}(q-1)^n$ where $K_0(\p)=\O^\times K_1(\p)$ is the 
  Iwahori subgroup,
for $\sigma$ of conductor $q^c$ he obtained the formal degree
\[d_\sigma^C = b\frac{q^n-1}{q^b-1}q^{\frac12((n-1)(c-n)+b-n)},\]
where $b=\gcd(c,n)$.
Taking $n=2$ and setting $\meas(\ol K)=1$, we have
\[\meas(\ol{K_0(\p)}) = [\ol K:\ol{K_0(\p)}]^{-1}=(q+1)^{-1} =\frac2{q-1}\meas_C(\ol{K_0(\p)}).\]
So in our normalization, his result gives
\[
d_\sigma = \frac{q-1}2d_\sigma^C=\frac{q-1}2\cdot b\frac{q^2-1}{q^b-1}q^{\frac12(c+b-4)},
\]
where $b=\gcd(c,2)$.
\end{proof}

Now fix three integers $M,T,S$ which are pairwise relatively prime with $S$ and $T$ square-free, and write
\begin{equation}\label{MST}
N=M \prod_{p|S}p^{2r_p} \prod_{p|T}p^{2r_p+1} 
\end{equation}
where $r_p\ge 1$ for each $p\mid ST$.  Fix a tuple $\pi_{ST}=(\pi_p)_{p|ST}$
  of supercuspidal representations of $\olG(\Q_p)$,
  with $\pi_p$ of conductor $p^{2r_p}$ 
  (resp.\ $p^{2r_p+1}$) if $p\mid S$ (resp.\ $p\mid T$).
For each $p\mid ST$, we may write $\pi_p=\cInd_{J_p}^{G(\Q_p)}(\lambda_p)$, where
$J_p$ is an open subgroup which contains, and is compact modulo, $Z_p=Z(\Q_p)$, and 
 $\lambda_p$ is an irreducible representation which is trivial on $Z_p$, subject to:
\begin{itemize}
\item $\dim\lambda_p=1$ if $p\mid T$
\item $J_p=Z_pK_p$ if $p\mid S$.
\end{itemize}
We will define a $Z(\A)$-invariant test function $f:G(\A)\longrightarrow \C$ with the property
\[\tr R(f) = \dim S_k^\new(N;\pi_{ST}),\]
where $R(f)$ is the operator on $L^2(\olG(\Q)\bs\olG(\A))$ defined by
\[R(f)\phi(x)=\int_{\olG(\A)}f(g)\phi(xg)dg.\]

For $p\mid ST$, define a local test function $f_{\pi_p}: G(\Q_p)\longrightarrow \C$ as follows:
\begin{equation}\label{fp}
f_{\pi_p}(x)=\begin{cases}d_{\pi_p}\ol{\lambda_p(x)}&\text{if } x\in J_p \text{ and } p \mid T\\
\ol{\tr \lambda_p(x)}&\text{if } x\in J_p \text{ and } p \mid S\\
0 &\text{if } x \not \in J_p,\end{cases}
\end{equation}
where $d_{\pi_p}$ is the formal degree of $\pi_p$, given in \eqref{fd}.
Thus, $f_{\pi_p}$  is the complex conjugate of a matrix coefficient of $\pi_p$ when $p \mid T$ 
and the sum of $\dim\lambda_p=d_{\pi_p}$ (conjugated) matrix coefficients when $p \mid S$.

At primes $\ell\nmid N$, we let $f_\ell$ be the characteristic function of $Z_\ell K_\ell$.

Let $q$ be a prime dividing $M$, and let $N_q=v_q(N)$. For $0\le j\le N_q$ let 
\[K_0(q^j)_q=\{\mat abcd\in K_q|\, c\in q^j\Z_q\}\]
as usual, and define $\phi_{q^j}:G(\Q_q)\longrightarrow\C$ by
\[\phi_{q^j}(g)=\begin{cases}\meas(K_0(q^j)_q)^{-1}&\text{if }g\in Z_qK_0(q^j)_q\\
0&\text{otherwise.}\end{cases}\]
Used as a local test function in the trace formula, $\phi_{q^j}$ serves to project the
automorphic spectrum onto its $K_0(q^j)_q$-fixed vectors. 
We define $f_q$ to be the linear combination
\begin{equation}\label{fq}
f_q = \phi_{q^{N_q}}-2\phi_{q^{N_q-1}}+\phi_{q^{N_q-2}},
\end{equation}
where $\phi_{q^j}$ is taken to be identically $0$ if $j<0$.
The role of $f_q$ is to give the trace on the locally $q^{N_q}$-new part of the spectrum.
Indeed, using \cite[Corollary to the proof]{Cas},
one shows that for any infinite dimensional irreducible admissible representation $\pi_q$ of $\PGL_2(\Q_q)$,
   $\tr \pi_q(f_q)\in \{0,1\}$ is nonzero if and only if $\operatorname{cond}(\pi_q) =q^{N_q}$.

At $\infty$ we take $f_\infty=d_k\ol{\sg{\pi_k(g)v,v}}$, where $\pi_k$ is the
weight $k$ discrete series representation of $\olG(\R)$, $v$ is a lowest weight unit vector,
and $d_k$ is the formal degree computed relative to our fixed choice of Haar measure on $\olG(\R)$.
This function is integrable over $\olG(\R)$ if and only if $k>2$, so this hypothesis will be
in force.

Let 
\[f=f_{\pi_{ST},k,N}=f_\infty \prod_{p|ST}f_{\pi_p}\prod_{q| M}f_q\prod_{\ell\nmid N}f_\ell\in L^1(\olG(\A)).\]
By \cite[Prop. 5.5 and Theorem 7.1]{K}, whose proofs carry over verbatim to the case with extra level
 structure $M$ included here, assuming $k>2$ 
and also that $T\ge 5$ is odd, we have
\begin{equation}\label{dim1}
\dim S_k^{\new}(N;\pi_{ST}) = \tr R(f)= \frac\pi 3f(1)
+ \frac12\Phi(\mat{}{-T}1{},f),
\end{equation}
for the elliptic orbital integral
\begin{equation}\label{Phig}
\Phi(\g,f)=\int_{\ol{G_\g(\Q)}\bs\olG(\A)}f(g^{-1}\g g)dg,
\end{equation}
where $G_\g$ is the centralizer of $\g$ in $G$.
(There are up to two additional elliptic terms if $T\in\{1,3\}$ or $T$ is even.)

The orbital integral factorizes as the product of a global measure term and the local orbital
integrals
\[\Phi(\g,f_{p})=\int_{\ol{G_\g(\Q_p)}\bs\olG(\Q_p)}f_{p}(g^{-1}\g g)dg\]
for $p\mid 2N$ (cf. \eqref{Phiglobal} below).
(We must fix appropriate Haar measures on the local groups $\ol{G_\g(\Q_p)}$. See \cite[\S3.3]{K} for details.)
The difficulty in evaluating \eqref{dim1} lies in computing $\Phi(\smat{}{-T}1{},f_{\pi_p})$
for each $p\mid ST$.
In the case where $\pi_p$ has conductor $p^2$ or $p^3$, this was done in \cite{K}. The 
general case is considerably more difficult, and computing $\Phi(\g,f_{\pi_p})$ for general 
conductor and $\g$ remains open.
However, simply by considering the support of $f_{\pi_p}$ we obtain the following,
which extends \cite[Prop. 5.6]{K}.

\begin{proposition}\label{Tprop}
Fix $N$ as in \eqref{MST}, with $T\ge 5$ odd.
For each $p\mid T$, let $E_p/\Q_p$ be the ramified quadratic extension attached to $\pi_p$.
Suppose that either
\begin{enumerate}
\item[(i)] for some $p\mid T$, $E_p\neq \Q_p(\sqrt{-T})$, or
\item[(ii)] for some odd $p\mid S$, $\Bigl(\frac{-T}p\Bigr)=1$.
\end{enumerate}
 Then
\begin{equation}\label{Tdim}
\dim S_k^\new(N;\pi_{ST}) = \frac{k-1}{12}\psi^\new(M)\prod_{p|T}\frac{p^2-1}2p^{r_p-1}
\prod_{p|S}(p-1)p^{r_p-1}
\end{equation}
for $\psi^\new(M)$ as defined in Theorem \ref{dim}.
\end{proposition}
\begin{remark} The argument below shows that more generally without the hypotheses on $T$, 
  when $p\mid T$ and $\g\in G(\Q)$ is elliptic in $G(\Q_p)$ with $v_p(\det\g)$ odd,
$\Phi(\g,f_{\pi_p})\neq 0$ only when $E_p = \Q_p(\g)$.
\end{remark}

\begin{proof}
The identity term in \eqref{dim1} is given by
\[\frac\pi3f(1)=\frac\pi3 d_k \prod_{q|M}f_q(1) \prod_{p|ST}d_{\pi_p}.\]
Let $\psi$ be the multiplicative function defined on prime powers by
\[\psi(q^j) =[K_q:K_0(q^j)_q]=\meas(K_0(q^j)_q)^{-1}=q^j(1+\frac1q)\]
for $j>0$.  Then by \eqref{fq} for a prime $q|M$ we have
\[f_q(1)=\psi(q^{N_q}) - 2\psi(q^{N_q-1})\delta_{N_q\ge 1}+\psi(q^{N_q-2})\delta_{N_q\ge 2},\]
where $\delta$ is an indicator function.
One checks easily that this coincides
with $\psi^\new(q^{N_q})$. 
Using \eqref{fd} and the fact (see, e.g., \cite[Prop. 14.4]{KL}) that $d_k=\frac{k-1}{4\pi}$,
we see that the identity term coincides with the right-hand side of \eqref{Tdim}.

By \eqref{dim1}, it remains to show that $\Phi(\g,f)=0$ under the given hypothesis, 
where $\g=\smat{}{-T}1{}$.
If hypothesis (ii) holds, then the characteristic polynomial $P_\g(X)=X^2+T$ factors 
 mod $p$, so by Hensel's Lemma (using $p\neq 2$) $\g$ is hyperbolic (rather than elliptic) 
in $G(\Q_p)$.  This implies that $\Phi(\g,f_{\pi_p})=0$ (see \cite[Prop. 4.3]{K}), giving
the result in this case.

 Now suppose that $\Phi(\g,f)\neq 0$.
Then for each $p\mid T$, $f_p(g^{-1}\g g)\neq 0$ for some $g\in G(\Q_p)$.  
Write $U_p^{r_p}$ for $U^{r_p}$ in \eqref{Ur}, and similarly $t_p$ and $\chi_p$ for the $t$ and $\chi$ determined
by $\pi_p$ as in Proposition \ref{chiprop}.  Then the support of $f_p$ satisfies
\[J_{E_p,r_p}=E_p^\times U_p^{r_p} \subset E_p^\times U_p^1=g_{\chi_p}\Q_p^\times U_p^1 \bigcup \Q_p^\times U_p^1.\]
Since $\det(g^{-1}\g g)=T\in p\Z_p^\times$, $g^{-1}\g g$ must lie in the ramified component 
$g_{\chi_p}\Q_p^\times U^1_p$:
\[g^{-1}\g g = z\mat{}{t_p}p{}\mat ab{pc}d \in \Q_p^\times g_{\chi_p} U_p^1,\]
where in fact $z\in \Z_p^\times$.
Taking determinants and dividing by $p$,
\[T/p \equiv -z^2 t_p\mod p,\]
which shows that $(\frac{-T/p}p)=(\frac{t_p}p)$. Since $E_p=\Q_p(\sqrt{pt_p})$ 
this condition is equivalent to $E_p=\Q_p(\sqrt{-T})$.
It follows that under hypothesis (i), $\Phi(\g,f)=0$.
\end{proof}

\subsection{Invariance of global dimension across local Galois orbits}\label{invariance}

Here we use Proposition~\ref{prop:locgal} to show that the number of newforms with fixed ramified 
supercuspidal components $\pi_p$ at a finite set of primes $p \ge 5$ depends only on the
Galois orbit of each $\pi_p$.  

\begin{proposition}\label{prop:dimeq}
Let $k \ge 2$ and $T\ge 5$ a square-free odd integer.  Let $M\ge 1$ be an integer 
relatively prime to $T$, and write $N = M\prod_{p|T}p^{2r_p+1}$ for $r_p\ge 1$, with $r_3=1$ if $3\mid T$.
For each $p \mid T$, let $\pi_p$ and $\pi_p'$ be irreducible supercuspidal representations of 
$\PGL_2(\Q_p)$ of conductor $2r_p+1$ such that $E_{\pi_p} = E_{\pi_p'}$ and 
$\varepsilon_{\pi_p} = \varepsilon_{\pi_p'}$.
Write $\pi_T = (\pi_p)|_{p | T}$ and $\pi'_T = (\pi'_p)|_{p | T}$.
Then
\[ \dim S_k^{\new}(N;\pi_T) = \dim S_k^{\new}(N;\pi_T'). \]
\end{proposition}
\begin{remark}
We allow $k=2$ since our proof does not rely on the trace formula.
\end{remark}

\begin{proof}
By Proposition~\ref{prop:locgal}, for every $p \mid T$ we have
$\pi_p' \simeq \pi_p^{\sigma_p}$ for some $\sigma_p \in \Gal(\Q(\zeta_{p^{r_p}})^+/\Q)$.  
Let $n = \prod_{p \mid T} p^{r_p}$.  Noting that $\Q(\zeta_{a})\cap \Q(\zeta_b) = \Q(\zeta_{\gcd(a,b)})=\Q$
if $(a,b)=1$,
we have
\[\Gal(\Q(\zeta_n)/\Q) \simeq \prod_{p \mid T} \Gal(\Q(\zeta_{p^{r_p}})/\Q).\] 
Hence there exists $\sigma \in \Gal(\Q(\zeta_n)/\Q)$ which has image $\sigma_p$ in each 
quotient $\Gal(\Q(\zeta_{p^{r_p}})^+/\Q)$.  In particular, $\pi_p^\sigma=\pi_p'$ for each $p|T$.

Now let $f$ be a normalized Hecke eigenform in $S_k^{\new}(N;\pi_{T})$.
The automorphic representation attached to $f$ has the form
   $\pi = \pi_k\otimes \bigotimes_p \pi_p$. The newform $f^\sigma$ (see \S\ref{Galdec})
 corresponds to the automorphic representation $\pi_k\otimes \bigotimes_p\pi_p^\sigma$.  This follows
from Proposition \ref{GaloisSatake}, together with strong multiplicity-one and the fact that the latter 
is known to be a cuspidal automorphic representation of $\GL_2(\A_\Q)$ (\cite[Section I.8]{wald}).
Hence $\sigma$ defines a vector space isomorphism of $S_k^{\new}(N;\pi_{T})$ 
with $S_k^{\new}(N;\pi_{T}^\sigma)=S_k^{\new}(N;\pi_T')$.  In particular these spaces have 
the same dimension.
\end{proof}

\subsection{Proof of Theorem \ref{dim}}\label{dimproof}

\begin{lemma}\label{Phip}
Fix a tuple $\pi_T=(\pi_p)_{p|T}$ as in Theorem \ref{dim}.
Define
\[\Phi_T = \prod_{p|T} \Phi(\smat{}{-T}1{},f_{\pi_p}).\]
Then there exists $I_T\in\C$ depending only on the $r_p$ such that
\[\Phi_T  =\Delta(\pi_T)\e_{\pi_T}I_T,
\]
with notation as in Theorem \ref{dim}.
Consequently \eqref{dim1} becomes
\begin{equation}\label{dim2}
\dim S_k^{\new}(N;\pi_T) = \frac{k-1}{12}\psi^\new(M)\prod_{p|T}\frac{p^2-1}2p^{r_p-1} 
+ \Delta(\pi_T)\e_{\pi_T}A,
\end{equation}
where $A$ is a constant depending only on $k$, $N$ and $T$, and independent of the choice of $\pi_T$.
\end{lemma}
\begin{remark}
 When $p\mid T$, the definition of $\Phi(\g,f_{\pi_p})$ entails a choice of Haar measure on the compact
group $\ol{G_\g(\Q_p)}$.  The choice is immaterial here but for concreteness we normalize to give
it measure $1$, and consequently
\[\Phi(\g,f_{\pi_p})=\int_{\olG(\Q_p)}f_{\pi_p}(g^{-1}\smat{}{-T}1{}g)dg.\]
\end{remark}

\begin{proof}
By Proposition \ref{prop:dimeq} and \eqref{dim1}, $\Phi_T$ depends only on the $r_p$, 
the fields $E_p$, and the $\e_{\pi_p}$.
We already showed in Proposition \ref{Tprop} that $\Phi_{T}$ vanishes when $\Delta(\pi_{T})=0$.
So we may assume that $\Delta(\pi_{T})=1$, i.e., $E_p=\Q_p(\sqrt{pt_p})=\Q_p(\sqrt{-T})$ 
  for each $p\mid T$.

 Let $\g=\smat{}{-T}1{}$.  Since $\det\g=T$,
\begin{equation}\label{eout}
\Phi_{T} = \prod_{p|T}\int_{\olG(\Q_p)}f_{\pi_p}(g^{-1}\g g)dg
 = \e_{\pi_T}\prod_{p|T}\int_{C_p}f_{\pi_p}(g_{\chi_p}^{-1}g^{-1}\g g)dg
\end{equation}
by \eqref{epi},
where $C_p=\{g\in \olG(\Q_p)|\, g^{-1}\g g\in g_{\chi_p} \Q_p^\times \O_{E_p}^\times U_p^{r_p}\}$.
The value of the inducing character $\lambda_p$ (hence $f_{\pi_p}$ by \eqref{fp}) 
  on $\Q_p^\times\O_{E_p}^\times U^{r_p}_p$ is independent of $\e_p=\lambda_p(g_{\chi_p})$.
Thus, $I_T$ (represented by the product of the integrals over the $C_p$) depends only on the
conductors $p^{2r_p+1}$ of the $\pi_p$'s.

Globally we have
\begin{equation}\label{Phiglobal}
\Phi(\g,f) = \meas(\ol{G_\g(\Q)}\bs \ol{G_\g(\A)}) \Delta(\pi_{T})\e_{\pi_T}I_T
\prod_{\ell\nmid T}\Phi(\g,f_\ell).
\end{equation}
The result now follows from \eqref{dim1}, in view of the main term computed in Proposition
\ref{Tprop}.
\end{proof}

To prove Theorem \ref{dim}, we just need to compute $A$ from the above lemma.
By \eqref{orth},
\begin{equation}\label{Wp}
\tr(W_{T}|S_k^{\new}(N)) =\sum_{\pi_{T}: \e_{\pi_T}=1}\dim S_k^\new(N;\pi_{T})
-\sum_{\pi_{T}: \e_{\pi_T}=-1}\dim S_k^\new(N;\pi_{T}).
\end{equation}
The number of supercuspidal representations of $\PGL_2(\Q_p)$ with conductor $p^{2r_p+1}$ 
is $2p^{r_p-1}(p-1)$, exactly half of which have root number $+1$ (resp. $-1$), as described 
  in \S\ref{scmodel}.  
It follows easily that half of the tuples $\pi_T$ have $\e_{\pi_T}=+1$ (resp. $-1$), so 
applying the trace formula on the right-hand side of \eqref{Wp}, the main terms all cancel out.
Given $p|T$, of the possibilities for $\pi_p$, half, or $p^{r_p-1}(p-1)$, satisfy 
  the non-vanishing condition that $E_p=\Q_p(\sqrt{-T})$.  
So after eliminating the main terms in \eqref{Wp}, the
number of nonzero summands remaining is
\[\prod_{p|T}p^{r_p-1}(p-1).\]
By the above lemma, the nonzero terms that remain all have the same value, up to the sign $\e_{\pi_T}$
which is eliminated by the subtraction in \eqref{Wp}.  Thus, in the notation of the above lemma,
\[\tr(W_{T}|S_k^{\new}(N)) = A \prod_{p|T} p^{r_p-1}(p-1).\]
Solving for $A$, \eqref{main} of Theorem \ref{dim} follows from \eqref{dim2}.

Next, the calculations in \cite[Proposition 3.2]{M23} and \cite[\S 4.2]{M25} respectively, together with the 
minor corrections made in the Appendix below, give the following, with notation as in Theorem \ref{dim}:
\begin{itemize}
\item When $M=1$ and $N \ne 27$,  we have
\begin{equation}\label{WTM1}
 \tr(W_T|S_k^{\new}(N))  = {(-1)^{k/2}} b_{T,0}\, h(-T) \prod_{p|T}(p-1)p^{r_p-1}; 
\end{equation}
\item  When $T=p \ge 5$ is prime,
\begin{equation}\label{WTTp}
 \tr(W_p|S_k^{\new}(N)) = {(-1)^{k/2}}  (p - 1) p^{r-1} \, b_{p,v_2(M)} \, 
  \kappa_{-p}(M') \, h(-p).
\end{equation}
\end{itemize}
We note that the constant $b_{T,e}$ in 
Theorem \ref{dim} coincides with $c_{T,e}/2$, where $c_{T,e}$ is given by $\beta(N)$ in \cite[(1.1)]{M23} in the first case and by the class number coefficient in \cite[Table 2]{M25} in the second case.
The remaining parts of Theorem \ref{dim}, namely \eqref{M1} and \eqref{Tp}, follow immediately
upon substituting the above formulas into \eqref{main}.

\subsection{Extension to allow depth zero supercuspidals}\label{dz}

With the proof of Theorem \ref{dim} complete, we indicate here how it may be extended so as
to allow for prescribed depth zero supercuspidals at certain places.  First,
by \cite[(6.22),(5.6)]{K}, when $T\equiv 1\mod p$ and $\operatorname{cond}(\pi_p)=p^2$,
\[\Phi(\smat{}{-T}1{},f_{\pi_p}) = \begin{cases}2\e_p&\text{if }p\equiv 3\mod 4\\
  0&\text{if }p\equiv1\mod 4.\end{cases}\]
(When $p\equiv 1\mod 4$, we have $\left(\frac {-T}p\right)=\left(\frac{-1}p\right) =1$ and (ii) of
 Proposition \ref{Tprop} applies.)
 This allows us to extend Proposition \ref{prop:dimeq} to also include prescribed depth-zero supercuspidals
$\pi_p,\pi_p'$ at each $p\mid S$ as long as $\e_{\pi_p}=\e_{\pi_p'}$, 
   and $T\equiv 1\mod S$. 
One also needs to adjust for the fact that \eqref{orth} does not hold in this case. Instead we need
to consider $S$-minimal newforms, i.e., those whose level cannot be reduced at primes dividing $S$
by twisting.  Taking $N=M\prod_{p|T}p^{2r_p+1}\prod_{p|S}p^2$
with each $r_p\ge 1$, 
$S_k^{\text{\rm{S-min}}}(N)=\bigoplus_{\pi_{ST}}S_k^{\new}(N;\pi_{ST})$.

By similar (but slightly more involved) arguments to those that led to the proof of Theorem \ref{dim},
assuming $T\equiv 1 \pmod S$, $T\ge 5$ is odd, and $r_3=1$ if $3\mid T$, we find:
\begin{align}\label{main2}
\dim S_k^\new(N;\pi_{ST}) = \frac{k-1}{12}\psi^\new(M)&\prod_{p|T}\frac{p^2-1}2p^{r_p-1}
\prod_{p|S}(p-1)\\
\notag&+\Delta(\pi_{ST})\e_{\pi_{ST}}\frac{\tr(W_{ST}|S_k^{\text{\rm{S-min}}}(N))}
{\prod_{p|T}(p-1)p^{r_p-1}\prod_{\text{odd }p|S}\frac{p-1}2},
\end{align}
where $\Delta(\pi_{ST})\in\{0,1\}$ extends $\Delta(\pi_T)$ by assigning the value $0$ 
if $p\equiv 1\mod 4$ for some $p\mid S$, or if $T\equiv 3\mod 4$ and $S$ is even.

\begin{remark} By \cite[Theorem 6(ii)]{AL}, if $p \mid S$ and $f \in S_k^\new(N)$ is not 
$p$-minimal, then $W_p f = {-1 \leg p } f$.  Thus one can compute Atkin--Lehner  traces on 
$S$-minimal spaces from the traces on the full newspaces and dimensions of $d$-minimal 
subspaces for $d \mid S$.  For instance, when $S = p$ is prime, one has  
\[ \tr(W_{ST}|S_k^{\text{\rm{S-min}}}(N)) = \tr(W_{ST}|S_k^{\new}(N)) - {-1 \leg p}  
\left( \dim S_k^{\new}(N)) - \dim S_k^{\text{\rm{S-min}}}(N) \right). \]
One can compute the dimensions of $d$-minimal subspaces by subtracting away  
dimensions of non-minimal forms in a similar way to \cite{child}, which 
considers the ($N$-)minimal subspaces.
\end{remark}

\section*{Appendix: Errata to \cite{M23,M25}}

Here we correct two mathematical typographical errors 
in the formulas for $\tr(W_T|S_k^{\new}(N))$ from the published works \cite{M23,M25} 
that were used in the proof of Theorem~\ref{dim}.\footnote{Corrected versions of these papers are available at \url{https://arxiv.org/abs/2207.08121v3} and \url{https://arxiv.org/abs/2409.02338v3}.}

\begin{enumerate}
\item
We used \cite[Proposition 3.2]{M23} to explicate this trace when $M=1$, and to our knowledge that 
statement is correct.  However, the case of the proposition that we use here is also 
stated in \cite[Theorem 1.2]{M23}, but the definition of $\delta$ there has a misprint.  
It should read $\delta = 1$ if $(N_2, k) = (1,2)$ rather than if $(N,k) = (1,2)$.  
This agrees with \cite[Proposition 3.2]{M23}.

\item 
We used the calculations of \cite[\S 4]{M25} to explicate $\tr(W_T|S_k^{\new}(N))$ 
when $T=p$ is prime.
However, the equation just above \cite[Prop. 4.3]{M25} should read:
\begin{align*}
\aleph &= c_1 \left( \eta_{\Delta_0}(q^{\frac{r-1}2}) - 2 \eta_{\Delta_0}(q^{\frac{r-3}2}) + \delta_{r \ge 5} \eta_{\Delta_0}(q^{\frac{r-5}2}) \right) h'(\Delta_0) \\
&= c_1 \left(\sigma(q^{\frac{r-1}2}) - 2\sigma(q^{\frac{r-3}2}) + \delta_{r \ge 5} \sigma(q^{\frac{r-5}2}) \right) h'(\Delta_0).
\end{align*}
In \emph{loc.\ cit.}, the factor of $2$ in the middle terms on the right 
mistakenly appears inside the arguments of $\eta_{\Delta_0}$ and $\sigma$.  
The factor of 2 appears in the correct location in the definition of $\aleph$ several lines earlier.

In context, $q$ is a prime and $r\ge 3$ is an odd integer, so in fact this simplifies to
\[ \aleph = c_1 \left(q^{\frac{r-1}2}  - q^{\frac{r-3}2} \right) h'(\Delta_0). \]
This does not affect the proof of Proposition 4.3 or Theorem 1.1 in \cite{M25}.
\end{enumerate}

\end{document}